\documentclass[11pt]{article}

\usepackage{amsmath, amsfonts, amsthm, amssymb,xcolor,graphicx}
\usepackage[shortlabels]{enumitem}
\usepackage{float}
\usepackage[initials]{amsrefs}
\usepackage[bookmarks]{hyperref}
\usepackage{comment}
\usepackage{a4wide}

\newtheorem{thm}{Theorem}[section]
\newtheorem{lma}[thm]{Lemma}
\newtheorem{cor}[thm]{Corollary}

\newtheorem{defn}[thm]{Definition}

\renewcommand{\i}{\mathtt{i}  }
\renewcommand{\j}{\mathtt{j}  }
\renewcommand{\k}{\mathtt{k}  }
\renewcommand{\l}{\mathtt{l}  }
\newcommand{\I}{\mathcal{I}}
\newcommand{\N}{\mathbb{N}}
\renewcommand{\a}{\mathbf{a}}
\renewcommand{\b}{\mathbf{b}}
\renewcommand{\c}{\mathbf{c}}
\newcommand{\R}{\mathbb{R}}
\newcommand{\ud}{\overline{\dim}_{\mathrm{loc}}}
\newcommand{\ld}{\underline{\dim}_{\mathrm{loc}}}
\newcommand{\locd}{\dim_{\mathrm{loc}}}

\title{Ledrappier-Young formulae for a family of nonlinear attractors}
\date{\vspace{-5ex}}

\author{Natalia Jurga \& Lawrence D. Lee }

\begin{document}

\maketitle

\begin{abstract} 
We study a natural class of invariant measures supported on the attractors of a family of nonlinear, non-conformal iterated function systems introduced by Falconer, Fraser and Lee. These are pushforward quasi-Bernoulli measures, a class which includes the well-known class of Gibbs measures for H\"older continuous potentials. We show that these measures are exact dimensional and that their exact dimensions satisfy a Ledrappier-Young formula.


\end{abstract}

\section{Exact dimensionality and Ledrappier-Young formulae}

Given an iterated function system (IFS), which in this article will refer to a finite family of uniform contractions $\{S_i: [0,1]^2 \to [0,1]^2\}_{i=1}^N$, it is well-known that there exists a unique, non-empty, compact set $F \subseteq [0,1]^2$ such that $F=\bigcup_{i=1}^N S_i(F)$ which we call the attractor of the IFS. We say that a measure $\mu$ supported on $F$ is invariant (respectively ergodic) if there exists a $\sigma$-invariant (respectively ergodic) measure $m$ on $\Sigma=\{1, \ldots, N\}^\N$ (where $\sigma$ denotes the left shift map) such that $\mu=m \circ \Pi^{-1}$ where $\Pi: \Sigma \to [0,1]^2$ is the canonical coding map defined by $\Pi(i_1,i_2 \ldots)=\lim_{n \to \infty} S_{i_1}\circ \cdots \circ S_{i_n}([0,1]^2)$.

Recall that the (upper and lower) local dimensions of $\mu$ at $x$ are defined as
$$ \ud(x):=\limsup_{r \to 0} \frac{\log(\mu(B(x,r)))}{\log r}  \;\; \textnormal{and} \;\;  \ld(x):=\liminf_{r \to 0} \frac{\log(\mu(B(x,r)))}{\log r}$$
where $B(x,r)$ denotes a ball of radius $r$ centred at $x$. If $\ud(x)=\ld(x)$ we denote the common value by $\locd(x)$ and call it the local dimension of $\mu$ at $x$. If the local dimension exists and is constant for $\mu$ almost all $x$ we say that the measure $\mu$ is exact dimensional and call this constant the exact dimension of $\mu$, which we will denote by $\dim \mu$. In this case, most well-known notions of dimension coincide with the exact dimension of $\mu$.

In the dimension theory of measures, it is a problem of central interest to establish the exact dimensionality of ergodic invariant measures supported on attractors of IFS, and to provide a formula for the exact dimension. In many settings, the exact dimension has been shown to satisfy a formula in terms of Lyapunov exponents, various notions of entropy and the dimensions of projected measures. 

In particular, if the maps $S_i$ are conformal and the IFS satisfies an additional separation condition, it is a classical result that any ergodic invariant measure $\mu$ supported on the attractor $F$ is exact dimensional and its exact dimension is given by its measure-theoretic entropy over the Lyapunov exponent (see e.g. \cite{bedford}). In the substantially more difficult case where no separation condition is assumed, Feng and Hu generalised this classical result by showing that any ergodic invariant measure $\mu$ supported on the attractor $F$ is exact dimensional and its exact dimension is given by the projection entropy over the Lyapunov exponent. In this sense, exact dimensionality is understood in the conformal setting.

On the other hand, the question of whether every ergodic invariant measure supported on the attractor of a non-conformal IFS is exact dimensional is still very much open, and this question has recently received a lot of attention in the particular case where the maps $S_i$ are all affine. Feng \cite{feng} has very recently shown that all ergodic invariant measures supported on the attractors of IFS composed of affine maps are exact dimensional and satisfy a formula in terms of the Lyapunov exponents and conditional entropies. This answered a folklore open question in the fractal community and unified previous partial results obtained in \cite{bk,fjj}. In the non-conformal setting, this formula for the exact dimension of $\mu$ is often called a ``Ledrappier-Young formula'', following the work of Ledrappier and Young on the dimension of invariant measures for $C^2$ diffeomorphisms on compact manifolds \cite{ly1,ly2}.

While Feng's result settles the case of ergodic measures supported on self-affine sets, the more general case of ergodic measures supported on attractors of more general (i.e. nonlinear) non-conformal IFS is still open. In fact, the only result in this direction that the authors are aware of is \cite[Theorem 2.11]{feng}, where Feng and Hu prove exact dimensionality of ergodic invariant measures supported on the attractors of IFSs which can be expressed as the direct product of IFSs composed of $C^1$ maps on $\R$. The fact that there is limited literature concerning the exact-dimensionality of measures supported on general non-conformal attractors reflects the wider challenge of understanding  the dimension theory of nonlinear non-conformal attractors, although this appears to be an area of growing interest \cite{fs, pesin, ffl}.

In this article we consider (pushforward) quasi-Bernoulli measures supported on the attractors of nonlinear, non-conformal IFS which were introduced in \cite{ffl}, and we show that these are exact dimensional and satisfy a Ledrappier-Young formula. In \S 2 we introduce the class of attractors and measures which will be studied and state our main result, Theorem \ref{ly theorem}. In \S 3 we recall some technical results which were proved in \cite{ffl} concerning the contractive properties of the maps in our IFS. \S 4 contains the proof of Theorem \ref{ly theorem}, which adapts an approach used in \cite{fjj}.

\section{Our setting and statement of results}\label{results}

Let $f_{i,x}$, $f_{i,y}$ denote partial derivatives of $f_i$ with respect to $x$ and $y$ respectively and $g_{i,x}$, $g_{i,y}$ denote partial derivatives of $g_i$ with respect to $x$ and $y$ respectively.

We will consider the following family of attractors which were introduced by Falconer, Fraser and Lee \cite[Definitions 1.1 and 3.1]{ffl}. 

\begin{defn}\label{ifsdef}
Suppose $\mathcal{I}$ is a finite index set with $|\mathcal{I}|\geq 2$. For each $i \in \I$ let $S_i:[0, 1]^2\rightarrow [0, 1]^2$ be of the form $S_i(a_1, a_2)= (f_i(a_1), g_i(a_1, a_2))$, where:
\begin{enumerate}
\item  $f_i$ and $g_i$ are $C^{1+\alpha}$ contractions $(\alpha>0)$ on $[0, 1]$ and $[0, 1]^2$ respectively.
\item $\{S_i\}_{i \in \I}$ satisfies the \textbf{strong separation condition} (SSC) : the sets $\lbrace S_{i}(F)\rbrace_{i\in\mathcal{I}}$ are pairwise disjoint. 
\item $\lbrace S_{i}\rbrace_{i\in\mathcal{I}}$ satisfies  the \textbf{domination condition}: for each $i \in \I$
\begin{equation}\label{Domination} 
\inf_{{\bf a}\in[0, 1]^2}|f_{i,x}({\bf a})|>\sup_{{\bf a}\in[0, 1]^2}|g_{i,y}({\bf a})|\geq \inf_{{\bf a}\in[0, 1]^2}|g_{i,y}({\bf a})|\geq d,  
\end{equation}
where $d>0$.
\end{enumerate}
 Let $F$ denote the attractor of $\{S_i\}_{i \in \I}$.
\end{defn}

For $n\in\mathbb{N}$ we write $\I^n$ to denote the set of all sequences of length $n$ over $\mathcal{I}$ and we let $\I^*=\bigcup_{n\geq 1}\mathcal{I}^n$ denote the set of all finite sequences over $\I$. We let $\Sigma=\I^\N$ denote the set of infinite sequences over $\I$ and for $\i=(i_1, i_2,\dots)\in\Sigma$ we write $\i|n= (i_1, i_2,\dots, i_n)\in\I^n$ to denote the restriction of $\i$ to its first $n$ symbols. For $\i=(i_1, i_2,\dots, i_n)\in\I^n$ we write $S_{\i}= S_{i_1}\circ\cdots\circ S_{i_n}$ and we write $[\i]\subseteq\Sigma$ to denote the cylinder set corresponding to $\i$, which is the set of all infinite sequences over $\I$ which begin with $\i$.

\begin{defn}\label{quasibernoulli} 
We say that a measure $m$ on $\Sigma$ is quasi-Bernoulli if there exists some $L>0$ such that for all $\i, \j\in\I^*$
\begin{equation}\label{quasibernoulli1}
L^{-1}m([\i])m([\j])\leq m([\i\j])\leq L m([\i])m([\j]).
\end{equation}
\end{defn}
We will study the pushforward measure $\mu=m \circ \Pi^{-1}$ for a quasi-Bernoulli measure $m$, noting that $\mu$ is supported on $F$. Apart from including the important class of Bernoulli measures, quasi-Bernoulli measures also include the well-known class of Gibbs measures for H\"{o}lder continuous potentials. Furthermore it was shown in \cite{bkm} that this inclusion is strict.

The Shannon-McMillan-Breiman theorem allows us to define the entropy of $\mu$.

\begin{defn}[Entropy]\label{entropy}
There exists a constant $h(\mu) \leq 0$ such that for $m$-almost all $\i \in \Sigma$,
\begin{equation}\label{entropy1}
h(\mu)=\lim_{n \to \infty} \frac{1}{n} \log m([\i|n]).
\end{equation}
We call $h(\mu)$ the entropy of $\mu$.
\end{defn}

Apart from entropy the other key features of the Ledrappier-Young formula are the Lyapunov exponents, which describe the typical contraction rates in different directions. 

Lyapunov exponents are defined in terms of Jacobian matrices of the maps $S_{\i|n}$. We denote the Jacobian matrix of $S_{\i|n}$ at a point $\a\in [0, 1]^2$ by $D_\a S_{\i|n}$. We recall that the singular values of an $n\times n$ matrix $A$ are defined to be the positive square roots of the eigenvalues of $A^TA$, where $A^T$ denotes the transpose of $A$. The Lyapunov exponents are defined in terms of singular values of the matrices $D_\a S_{\i|n}$. For fixed $\a \in [0,1]^2$ and any $n\in\mathbb{N}$ we let $\alpha_1(D_\a S_{\i|n})\geq\alpha_2(D_\a S_{\i|n})$ denote the singular values of $D_\a S_{\i|n}$.  The sub-additive ergodic theorem then allows us to make the following definition. 

\begin{defn}[Lyapunov exponents]\label{lyapunov}
There exist constants $ \chi_2(\mu) \leq \chi_1(\mu)< 0$ such that for $m$ almost all $\i \in \Sigma$,
\begin{equation}\label{lyapunov1}
\chi_1(\mu)=\lim_{n \to \infty}\frac{1}{n} \log \alpha_1\left(D_{\Pi(\sigma^n \i)} S_{\i|n}\right)
\end{equation}
and
\begin{equation}\label{lyapunov2}
\chi_2(\mu)=\lim_{n \to \infty}\frac{1}{n} \log \alpha_2\left(D_{\Pi(\sigma^n \i)} S_{\i|n}\right).
\end{equation}
We call $\chi_1(\mu), \chi_2(\mu)$ the Lyapunov exponents of the system with respect to $\mu$.
\end{defn}

Let $\pi:[0,1]^2 \to [0,1]$ denote projection to the $x$-co-ordinate. Let $\pi(\mu)$ denote the projected measure which is supported on $\pi F$, which is the attractor of the (possibly overlapping) conformal IFS $\{f_i\}_{i \in \I}$ on $[0,1]$. Note that by \cite[Theorem 2.8]{fh}, $\pi(\mu)$ is exact dimensional. We denote its exact dimension by $\dim \pi(\mu)$.

We are now ready to state our main result.

\begin{thm}\label{ly theorem}
Let $\mu$ be a pushforward quasi-Bernoulli measure supported on $F$, where $F$ satisfies Definition \ref{ifsdef}. Then $\mu$ is exact dimensional and moreover its exact dimension $\dim \mu$ satisfies the following Ledrappier-Young formula
\[
\dim \mu = \frac{h(\mu)}{\chi_2(\mu)} + \frac{\chi_2(\mu)-\chi_1(\mu)}{\chi_2(\mu)} \dim \pi(\mu).
\]
\end{thm}

\section{Preliminaries}

Since each $f_i$ ($i \in \I$) only depends on the $x$-co-ordinate of a given point, it is easy to see that the Jacobian of each $S_i$ must be lower triangular. Denote the Jacobian by
$$D_\a S_i= \begin{pmatrix} f_{i,x}(\a) &0 \\ g_{i,x}(\a)& g_{i,y}(\a) \end{pmatrix}.$$
It is easy to see by the chain rule for any $\i\in\Sigma$ and $n\in\mathbb{N}$ the Jacobian of $S_{\i|n}$ must also be lower triangular which we will write as
$$D_\a S_{\i|n}= \begin{pmatrix} f_{\i|n,x}(\a) &0 \\ g_{\i|n,x}(\a)& g_{\i|n,y}(\a) \end{pmatrix}.$$
Our IFS has several useful properties, which were established in \cite{ffl}. To begin with we have the following bound, which allows us to control the off-diagonal entry:

\begin{lma}\label{fg bound} \cite[Lemma 3.3]{ffl}
There exists a constant $C>0$ such that for any $\i \in \Sigma$, $n \in \N$ and ${\bf a}, {\bf b}\in[0, 1]^2$,
\begin{equation} \label{g/f}
\frac{|g_{\i|n,x}(\a)|}{|f_{\i|n,x}(\b)|} \leq C.
\end{equation}
\end{lma}

A consequence of Lemma \ref{fg bound} is that the singular values of the the Jacobian matrices are comparable to their diagonal entries.  

\begin{lma}\label{Singular Value Bound}\cite[Lemma 3.4]{ffl}
There exists a constant $M\geq1$ such that for all ${\bf a}\in[0, 1]^2$, $\i\in\Sigma$ and $n\in\mathbb{N}$ the singular values of the Jacobian matrices $D_{{\bf a}}S_{\i|n}$ satisfy
\begin{equation}\label{comparability}
M^{-1} \ \leq\ \frac{\alpha_1\left(D_{{\bf a}}S_{\i|n}\right)}{|f_{\i|n, x}({\bf a})|},\  \frac{\alpha_2\left(D_{{\bf a}}S_{\i|n}\right)}{|g_{\i|n, y}({\bf a})|}\ \leq \ M.
\end{equation}
\end{lma}

Lemma \ref{Singular Value Bound}, together with the domination condition, implies that the two Lyapunov exponents are distinct, $\chi_2(\mu)< \chi_1(\mu)$.  

Another useful property of our IFS is that the diagonal entries of the Jacobian matrices satisfy a bounded distortion condition.

\begin{lma}\label{boundeddistortion}\cite[Lemma 3.2]{ffl}
There exists a constant $A\geq 1$ such that for all $\i\in\Sigma$, $n\in\mathbb{N}$ and all ${\bf a}, {\bf b}\in[0, 1]^2$,
\begin{equation}\label{ratiofandg}
A^{-1} \ \leq\ \frac{|f_{\i|n,x}({\bf a})|}{|f_{\i|n,x}({\bf b})|},\  \frac{|g_{\i|n,y}({\bf a})|}{|g_{\i|n,y}({\bf b})|}\ \leq \ A.
\end{equation}
\end{lma}

Finally, the singular values of the Jacobian matrices also satisfy bounded distortion.

\begin{lma}\label{svboundeddistortion}
There exists a constant $R\geq 1$ such that for all $\i\in\Sigma$, $n\in\mathbb{N}$ and all ${\bf a}, {\bf b}\in[0, 1]^2$,
\begin{equation}\label{ratiofalphas}
R^{-1} \ \leq\ \frac{\alpha_1\left(D_{{\bf a}}S_{\i|n}\right)}{\alpha_1\left(D_{{\bf b}}S_{\i|n}\right)},\  \frac{\alpha_2\left(D_{{\bf a}}S_{\i|n}\right)}{\alpha_2\left(D_{{\bf b}}S_{\i|n}\right)}\ \leq \ R.
\end{equation}
\end{lma}

\begin{proof}
Simply combine Lemma \ref{Singular Value Bound} and Lemma \ref{boundeddistortion}. 
\end{proof}

An easy but useful consequence of Lemma \ref{svboundeddistortion} is that the Lyapunov exponents defined in Definition \ref{lyapunov} may be expressed as
\begin{equation} \label{lyap-alt}
\chi_1(\mu)=\lim_{n \to \infty}\frac{1}{n} \log \alpha_1\left(D_{a_n} S_{\i|n}\right) \;\; \textnormal{and} \;\;\chi_2(\mu)=\lim_{n \to \infty}\frac{1}{n} \log \alpha_2\left(D_{a_n} S_{\i|n}\right)\end{equation}
for any sequence $(a_n)_{n\in\mathbb{N}}$ in $[0, 1]^2$, on the same set of $\i\in\Sigma$ of full $m$-measure that was used in Definition \ref{lyapunov}.

\section{Proofs}

The following key lemma allows us to estimate the $\mu$-measure of a small ``approximate square'' in $[0,1]^2$  by the product of the $m$-measure of an appropriate cylinder and the $\pi(\mu)$-measure of the $\pi$-projection of the ``blow up'' of the ``approximate square''. It is worth noting that this lemma is the only place where the assumption that $m$ is quasi-Bernoulli (Definition \ref{quasibernoulli}) is used.

For $r>0$, $n \in \N$, $\a\in[0, 1]^2$ and $\i \in \Sigma$ such that $\Pi(\i)=\a$ we write $B_n(\a, r)$  to denote the strip of points $\b=(b_1, b_2) \in S_{\i|n}([0,1]^2)$ whose $x$ co-ordinate satisfies $|b_1-a_1| \leq r/2$. We note that by the SSC, $\Pi$ is an injective map and therefore $B_n(\a,r)$ is well defined. For $x\in\mathbb{R}$ and $r>0$ we write $Q_1(x, r) = [x-\frac{r}{2}, x+\frac{r}{2}]$.

\begin{lma}\label{mu bound}
Let $r>0$, $n \in \N$, $\a=(a_1, a_2) \in F$ and $\i \in \Sigma$ such that $\Pi(\i)=\a$. Then
\begin{align}
&\mu(B_n(\a, r))\nonumber\\ 
&\leq Lm([\i|n]) \pi(\mu)\left(Q_1\left(\pi(\Pi(\sigma^n \i)), \frac{Mr}{\min_{\b \in [0,1]^2} \alpha_1\left(D_{{\bf b}}S_{\i|n}\right)}\right)\right)\label{mu bound 1}
\end{align}
and
\begin{align}
&\mu(B_n(\a, r))\nonumber\\
&\geq L^{-1}m([\i|n]) \pi(\mu)\left(Q_1\left(\pi(\Pi(\sigma^n \i)), \frac{M^{-1}r}{\max_{\b \in [0,1]^2}\alpha_1\left(D_{{\bf b}}S_{\i|n}\right)}\right)\right)\label{mu bound 2}
\end{align}
where $L$ is the constant from the quasi-Bernoulli property \eqref{quasibernoulli1} and where $M$ is as defined in Lemma \ref{Singular Value Bound}.
\end{lma}

\begin{proof}
Let 
\[
\mathcal{J} = \mathcal{J}(\i, n, r) = \left\{\j\in\mathcal{I}^* : S_{\i|n\j}([0,1]^2)\subseteq B_n(\a, r) \textnormal{ and } S_{\i|n\j^{\dagger}}([0,1]^2) \nsubseteq B_n(\a, r)  \right\}
\]
writing $\j^{\dagger}$ to denote $\j$ with the last symbol removed. It follows that
\[
\mu(B_n(\a, r)) = \sum_{\j\in\mathcal{J}}m([\i|n\j])
\]
and as $m$ is quasi-Bernoulli (Definition \ref{quasibernoulli}) we get
\begin{equation}\label{measureequation}
L^{-1}m([\i|n])\sum_{\j\in\mathcal{J}}m([\j])\leq\mu(B_n(\a, r))\leq Lm([\i|n])\sum_{\j\in\mathcal{J}}m([\j]).
\end{equation}
Note that the sets $\{S_{\j}([0, 1]^2)\}_{\j\in\mathcal{J}}$ are disjoint and exhaust $S_{\i|n}^{-1}(B_n(\a, r))$ in measure. Moreover, since $S_{\i|n}^{-1}B_n(\a, r)$ necessarily has height 1 we have
$$\sum_{\j\in\mathcal{J}}m([\j])=\mu(S_{\i|n}^{-1}B_n(\a, r))=\pi(\mu)(\pi S_{\i|n}^{-1}(B_n(\a, r))).$$
Observe that $\pi S_{\i|n}^{-1}(\a)=\pi(\Pi(\sigma^n(\i)))$. Writing $\hat{a}$ and $\hat{b}$ for the left and right endpoints of $\pi S_{\i|n}^{-1}(B_n(\a, r))$, it follows from the mean value theorem that
\[
|\pi(\Pi(\sigma^n \i))-\hat{a}| = \frac{|a_1-f_{\i|n}(\hat{a})|}{|f_{\i|n, x}(\hat{c}_1)|} = \frac{r}{2|f_{\i|n, x}(\hat{c}_1)|}
\]
for some $\hat{c}_1\in [0, 1]$ and 
\[
|\pi(\Pi(\sigma^n \i))-\hat{b}| = \frac{|a_1-f_{\i|n}(\hat{b})|}{|f_{\i|n, x}(\hat{c}_2)|} = \frac{r}{2|f_{\i|n, x}(\hat{c}_2)|}
\]
for some $\hat{c}_2\in [0, 1]$. By Lemma \ref{Singular Value Bound}
\begin{eqnarray*}
\max\left\{\frac{r}{2|f_{\i|n, x}(\hat{c}_1)|}, \frac{r}{2|f_{\i|n, x}(\hat{c}_2)|}\right\}&\leq& \frac{r}{2\min_{\b \in [0,1]^2} |f_{\i|n,x}(\b)|}\\
&\leq&\frac{Mr}{2\min_{\b \in [0,1]^2}\alpha_1\left(D_{{\bf b}}S_{\i|n}\right)}
\end{eqnarray*}
and 
\begin{eqnarray*}
\min\left\{\frac{r}{2|f_{\i|n, x}(\hat{c}_1)|}, \frac{r}{2|f_{\i|n, x}(\hat{c}_2)|}\right\}&\geq& \frac{r}{2\max_{\b \in [0,1]^2} |f_{\i|n,x}(\b)|}\\
&\geq& \frac{M^{-1}r}{2\max_{\b \in [0,1]^2}\alpha_1\left(D_{{\bf b}}S_{\i|n}\right)}.
\end{eqnarray*}
Therefore
\[
\sum_{\j\in\mathcal{J}}m([\j])\leq \pi(\mu)\left(Q_1\left(\pi(\Pi(\sigma^n \i)), \frac{Mr}{\min_{\b \in [0,1]^2} \alpha_1\left(D_{{\bf b}}S_{\i|n}\right)}\right)\right)
\]
and
\[
\sum_{\j\in\mathcal{J}}m([\j])\geq\pi(\mu)\left(Q_1\left(\pi(\Pi(\sigma^n \i)), \frac{M^{-1}r}{\max_{\b \in [0,1]^2}\alpha_1\left(D_{{\bf b}}S_{\i}\right)}\right)\right).
\]
Combining this with \eqref{measureequation} completes the proof.

\end{proof}

Recall by  \cite[Theorem 2.8]{fh} that as an ergodic measure on a self-conformal set, $\pi(\mu)$  is exact dimensional with exact dimension $\dim \pi(\mu)$. This informs us how $\pi(\mu)\left(Q_1\left(\pi(\Pi( \i)), \frac{1}{n}\right)\right)$ scales for an $m$-typical point $\i \in \Sigma$, although it does not provide any uniform bounds on the projected measure of this interval. The following lemma guarantees the existence of a set of positive measure on which we can uniformly bound $\pi(\mu)\left(Q_1\left(\pi(\Pi( \i)), \frac{1}{n}\right)\right)$.

\begin{lma}\label{log mu bound}
Let $\dim \pi(\mu)=t$. There exists a set $G\subseteq \Sigma$ with measure $m(G)\geq 1/2$ such that if $\varepsilon>0$, then for all $n$ sufficiently large
\[
\log\pi(\mu)\left(Q_1\left(\pi(\Pi( \i)), \frac{1}{n}\right)\right)\leq(t-\varepsilon)\log \left(\frac{1}{n}\right)
\]
and
\[
\log\pi(\mu)\left(Q_1\left(\pi(\Pi( \i)),\frac{1}{n}\right)\right)\geq(t+\varepsilon)\log \left(\frac{1}{n}\right)
\]
for all $\i \in G$.
\end{lma}

\begin{proof}
Define $f_n:\Sigma \rightarrow\mathbb{R}$ by
\[
f_n(\i)= \frac{\log\pi(\mu)\left(Q_1\left(\pi(\Pi( \i)),\frac{1}{n}\right)\right)}{-\log n }.
\]
Therefore for $m$ almost all $\i$
\[
\lim_{n\rightarrow\infty}f_n(\i)=\lim_{n\rightarrow\infty}\frac{\log\pi(\mu)\left(Q_1\left(\pi(\Pi( \i)), \frac{1}{n}\right)\right)}{-\log n}=t
\]
because $\pi(\mu)$ is exact dimensional. By Egorov's Theorem there exists a Borel measurable set $G\subseteq \Sigma$ with $m(G) \geq 1/2$ such that $f_n$ converges uniformly on $G$. In particular, for all $\varepsilon>0$ there exists $N_\varepsilon \in \N$ such that
$$t-\varepsilon \leq \frac{\log\pi(\mu)\left(Q_1\left(\pi(\Pi( \i)), \frac{1}{n}\right)\right)}{-\log n} \leq t+\varepsilon$$
for all $n \geq N_\varepsilon$ and $\i \in G$. 
Rearranging this expression yields the desired result.
\end{proof}

Next we show that for $m$-almost all $\i \in \Sigma$ the sequence of points $\{\sigma^n(\i)\}_{n \in \N}$ regularly visits the set $G$ from Lemma \ref{log mu bound}, yielding uniform bounds on the projected measure of the intervals that appear in \eqref{mu bound 1} and \eqref{mu bound 2} along a subsequence of $n \in \N$.

\begin{lma}\label{subsequence}
Let $\dim \pi(\mu)=t$ and for each $\i \in \Sigma$ let $(r_n(\i))_{n \in \N}$ be a positive null sequence such that $r_n(\i) \to 0$ uniformly over all $\i \in \Sigma$.  Then for $m$-almost all $\i\in \Sigma$ there exists a sequence $\{n_k\}_{k\in\mathbb{N}}$ such that for all $\varepsilon>0$
\begin{equation}\label{subsequence1}
\log\pi(\mu)\left(Q_1\left(\pi(\Pi(\sigma^{n_k}\i)), r_{n_k}(\i)\right)\right)\leq(t-\varepsilon)\log\left(r_{n_k}(\i)\right)
\end{equation}
and
\begin{equation}\label{subsequence2}
\log\pi(\mu)\left(Q_1\left(\pi(\Pi(\sigma^{n_k}\i)),r_{n_k}(\i)\right)\right)\geq(t+\varepsilon)\log\left(r_{n_k}(\i)\right)
\end{equation}
for all sufficiently large $k \in \N$.
Furthermore the sequence $\{n_k\}_{k\in\mathbb{N}}$ can be chosen to satisfy 
\[
\lim_{k\rightarrow\infty}\frac{n_{k+1}}{n_k} = 1.
\]
\end{lma}

\begin{proof}
Let $G$ be the set from the statement of Lemma \ref{log mu bound} and consider the characteristic function $\mathbf{1}_{G}$, which is easily seen to be in $L^1(\Sigma)$. We can now apply the Birkhoff Ergodic Theorem to obtain that for $m$-almost all $\i\in\Sigma$
\begin{align*}
\lim_{n\rightarrow\infty}\frac{1}{n}\sum_{j=0}^{n-1}\mathbf{1}_{G}(\sigma^j \i)=\int\mathbf{1}_{G} dm = m(G)\geq 1/2.
\end{align*}
This gives that for $m$-almost all $\i \in \Sigma$, $\sigma^j \i\in G$ with frequency greater than or equal to $1/2$. For each $\i\in\Sigma$ which satisfies this let $\{n_k\}_{k\in\mathbb{N}}$ be the sequence for which $\sigma^{n_k}\i\in G$ for all $k\in\mathbb{N}$. Then by Lemma \ref{log mu bound} for all $\varepsilon>0$ there exists $N_\varepsilon \in \N$ such that
\[
\log\pi(\mu)\left(Q_1\left(\pi(\Pi(\sigma^{n_k}\i)), \frac{1}{n}\right)\right)\leq\left(t-\frac{\varepsilon}{2}\right)\log\left(\frac{1}{n}\right)
\]
and
\[
\log\pi(\mu)\left(Q_1\left(\pi(\Pi(\sigma^{n_k}\i)), \frac{1}{n}\right)\right)\geq\left(t+\frac{\varepsilon}{2}\right)\log\left(\frac{1}{n}\right)
\]
for $n \geq N_\varepsilon$ and all $k \in \N$. Since $r_{n}(\i) \to 0$ uniformly over all $\i \in \Sigma$, we can choose $M_\varepsilon \in \N$ such that $r_n(\i) \leq \frac{1}{N_\varepsilon}$ for all $n \geq M_\varepsilon$. In particular for all $n_k \geq M_\varepsilon$ and $m$-almost all $\i \in \Sigma$ there exists $\ell \geq N_\varepsilon$ such that 
\[
\frac{1}{l+1}\leq r_{n_k}(\i)\leq\frac{1}{l},
\]
which gives
$$\frac{\log\pi(\mu)\left(Q_1\left(\pi(\Pi(\sigma^{n_k}\i)),\frac{1}{\ell}\right)\right)}{\log\left(\frac{1}{\ell}\right) +\log\left(\frac{\ell}{\ell+1}\right)}\leq \frac{\log\pi(\mu)\left(Q_1\left(\pi(\Pi(\sigma^{n_k}\i)),r_{n_k}(\i)\right)\right)}{\log r_{n_k}(\i)} $$
and
$$\frac{\log\pi(\mu)\left(Q_1\left(\pi(\Pi(\sigma^{n_k}\i)),r_{n_k}(\i)\right)\right)}{\log r_{n_k}(\i)} \leq \frac{\log\pi(\mu)\left(Q_1\left(\pi(\Pi(\sigma^{n_k}\i)),\frac{1}{\ell+1}\right)\right)}{\log\left(\frac{1}{\ell+1}\right) +\log\left(\frac{\ell+1}{\ell}\right)}.$$
Hence there exists $N_\varepsilon' \geq M_\varepsilon$ such that for $m$-almost all all $\i$ and all $n_k \geq N_\varepsilon'$,
\[\left|\frac{\log\pi(\mu)\left(Q_1\left(\pi(\Pi(\sigma^{n_k}\i)),r_{n_k}(\i)\right)\right)}{\log r_{n_k}(\i)}-t \right|\leq \varepsilon,\] giving the first part of the result. 

It remains to show that $\lim_{k\rightarrow\infty}n_{k+1}/n_k = 1$. Let $\i$ belong to the set of full $m$-measure for which $\sigma^j \i \in G$ with frequency at least $1/2$ and write
\[
S_{n_k}=\sum_{j=0}^{n_k-1}\mathbf{1}_{G}(\sigma^j \i).
\]
By Birkhoff's Ergodic Theorem $\lim_{k\rightarrow\infty}S_{n_k}/n_k = m(G)\geq 1/2$ and clearly $S_{n_{k+1}}=S_{n_k}+1$. Now note that
\[
\left|\frac{S_{n_k}}{n_k}\left(\frac{n_k}{n_{k+1}}-1\right)+\frac{1}{n_{k+1}} \right| = \left|\frac{S_{n_k}+1}{n_{k+1}}-\frac{S_{n_k}}{n_k}\right|\rightarrow 0
\]
as $k\rightarrow\infty$. As $S_{n_k}/n_k\rightarrow m(G)\geq 1/2$ and $1/n_{k+1}\rightarrow 0$ it follows that $n_k/n_{k+1} \rightarrow 1$ as $k\rightarrow\infty$, completing the result.
\end{proof}

The SSC gives us control over the distance between ``level $n$'' cylinders, as described in the following lemma.

\begin{lma} \label{sep-lemma}
There exists $\theta>0$ such that for any $\i, \l \in \Sigma$ with $\i|n \neq \l|n$,
\begin{equation}\label{sep}
\inf_{\a, \b \in [0,1]^2} d(S_{\i|n}(\a), S_{\l|_n}(\b)) \geq \theta \alpha_2(D_{\Pi(\sigma^{n-1} \i)}S_{\i|_{n-1}})
\end{equation}
where $d$ denotes the standard Euclidean metric.
\end{lma}

\begin{proof}
For notational convenience in this proof we shall sometimes write $a\lesssim b$  to mean that for $a, b\in\mathbb{R}$ we have $a\leq cb$ for some universal constant $c>0$, where $c$ is independent of any variable which $a$ and $b$ depend on.

We begin by showing that for any $a_1, b_1, b_2 \in [0,1]$ with $a_1\neq b_1$ and any $\i \in \I^*$,
\begin{equation} \label{boundedheight}
\frac{|g_\i(b_1,b_2)-g_\i(a_1,b_2)|}{|f_\i(b_1)-f_\i(a_1)|} \leq C
\end{equation}
where $C$ is the constant from \eqref{g/f}. To see this, notice that by the Mean Value Theorem there exist $c_1,c_2 \in (a_1,b_1)$ such that
$$\frac{|g_\i(b_1,b_2)-g_\i(a_1,b_2)|}{|f_\i(b_1)-f_\i(a_1)|}=\frac{|g_{\i,x}(c_1,b_2)||b_1-a_1|}{|f_{\i,x}(c_2)||b_1-a_1|}=\frac{|g_{\i,x}(c_1,b_2)|}{|f_{\i,x}(c_2)|} \leq C,$$
where the final inequality follows by \eqref{g/f}. 

Now, let $\a=(a_1,a_2), \b=(b_1,b_2) \in [0,1]^2$.  Define $\c=(b_1,a_2)$. We will now show that
\begin{equation}
d(S_\i(\a), S_\i(\b)) \gtrsim d(S_\i(\a),S_\i(\c))+d(S_\i(\c),S_\i(\b)), \label{foliation bound}
\end{equation}
where the implied constant is independent of $\i \in \I^*$, $\a$, and $\b$. To see this, we let $\gamma=|f_\i(a_1)-f_\i(b_1)|$, $\varepsilon=|g_\i(\a)-g_\i(\b)|$ and $\eta=|g(\a)-g(\c)|$. Note that $d(S_\i(\a), S_\i(\b))=\sqrt{\gamma^2+\varepsilon^2}$, $d(S_\i(\a),S_\i(\c))=\sqrt{\gamma^2+\eta^2}$. This is displayed visually in Figure \ref{distance figure}.
\begin{figure}[H]
  \centering
\includegraphics[width=\textwidth]{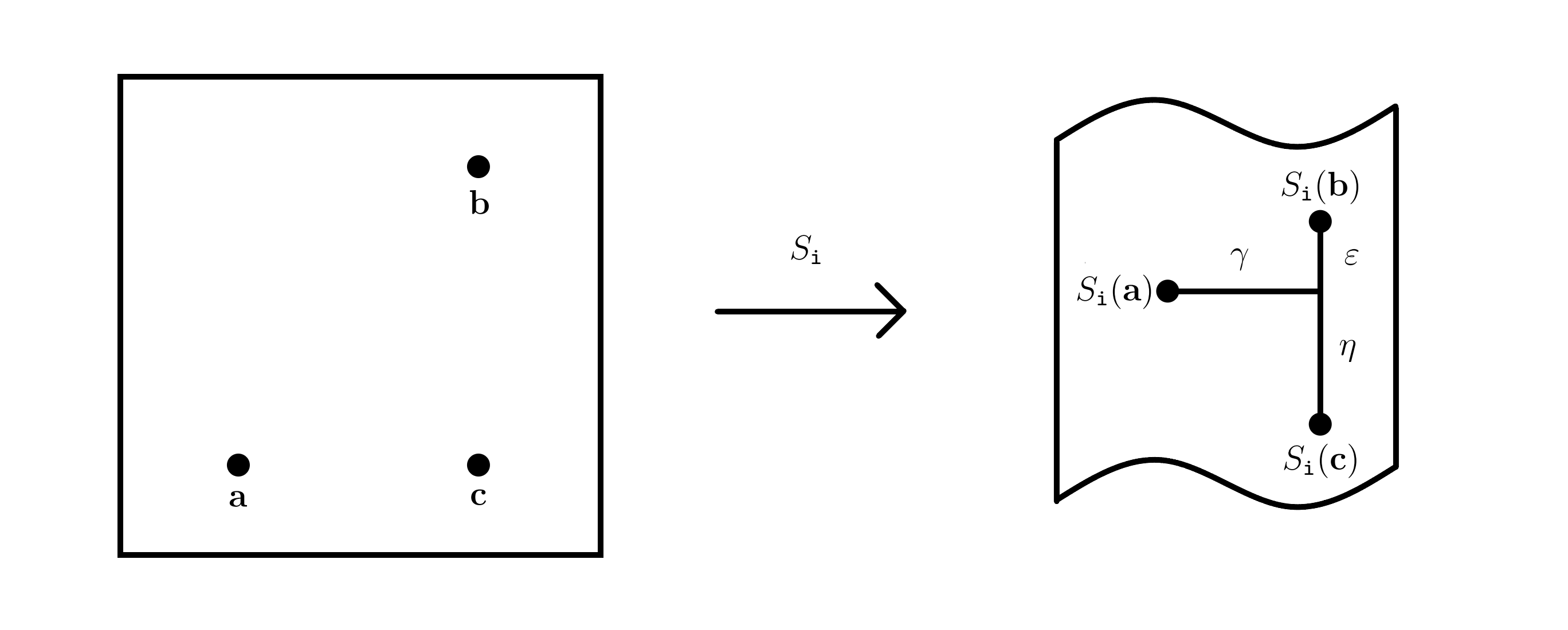}
\caption{The images of the points $\a, \b$ and $\c$ under $S_\i$ and the distances $\gamma, \epsilon$ and $\eta$.} 
\label{distance figure}
\end{figure}
There are now three possibilities: (i) $d(S_\i(\c),S_\i(\b))=\eta+\varepsilon$, (ii) $\eta>\varepsilon$ and $d(S_\i(\c),S_\i(\b))=\eta-\varepsilon$ or (iii) $\varepsilon>\eta$ and $d(S_\i(\c),S_\i(\b))=\varepsilon-\eta$. Hence
\begin{eqnarray}
\frac{d(S_\i(\a),S_\i(\c))+d(S_\i(\c),S_\i(\b))}{d(S_\i(\a), S_\i(\b))}&=& \frac{\sqrt{\gamma^2+\eta^2}+d(S_\i(\a),S_\i(\c))}{\sqrt{\gamma^2+\varepsilon^2}}. \label{quotient}
\end{eqnarray}
In cases (i) and (iii) we can use \eqref{boundedheight} to bound $\eta \lesssim \gamma$, yielding that 
$$\frac{\sqrt{\gamma^2+\eta^2}+d(S_\i(\a),S_\i(\c))}{\sqrt{\gamma^2+\varepsilon^2}} \lesssim \frac{\gamma+\varepsilon}{\sqrt{\gamma^2+\varepsilon^2}} \lesssim 1.$$
whereas in case (ii) we can use $\eta \lesssim \gamma$ to deduce that
$$\frac{\sqrt{\gamma^2+\eta^2}+d(S_\i(\a),S_\i(\c))}{\sqrt{\gamma^2+\varepsilon^2}} \lesssim \frac{\gamma}{\gamma} = 1.$$
This completes the proof of \eqref{foliation bound}.

Now, notice that by the Mean Value Theorem there exists $\c_1 \in [0,1]^2$ such that
\begin{eqnarray*}
d(S_\i(\a), S_\i(\c))^2 &=& f_{\i,x}(\c_1)^2|a_1-b_1|^2+g_{\i,x}(\c_1)^2|a_1-b_1|^2\\
&\geq& d(\a,\c)^2 f_{\i,x}(\c_1)^2 \geq d(\a,\c)^2 \sup_{\c_2 \in [0,1]^2} g_{\i,y}(\c_2)^2
\end{eqnarray*} 
by \eqref{Domination}. Similarly one can check that 
\[
d(S_\i(\b), S_\i(\c))^2 \geq  d(\b,\c)^2\inf_{\c_2 \in [0,1]^2}g_{\i,y}(\c_2)^2.
\] 
Therefore
\begin{eqnarray}
d(S_\i(\a),S_\i(\b)) &\gtrsim & d(S_\i(\a), S_\i(\c))+d(S_\i(\c), S_\i(\b)) \nonumber\\
&\gtrsim & (d(\a,\c)+d(\b,\c))\inf_{\c_2 \in [0,1]^2}|g_{\i,y}(\c_2)| \nonumber\\
&\geq& d(\a,\b)\inf_{\c_2 \in [0,1]^2}|g_{\i,y}(\c_2)|\nonumber \\
&\gtrsim& d(\a,\b)\sup_{\c_3 \in [0,1]^2}\alpha_2(D_{\c_3} S_\i) \label{cyl contr}
\end{eqnarray}
where the first inequality follows by \eqref{foliation bound} and the final one by Lemma \ref{Singular Value Bound}.

Finally, note that by the SSC, there exists $\delta>0$ such that 
\begin{equation}
\min_{i \neq j \in \I} \min_{x \in S_i(F)}\min_{y \in S_j(F)} d(x,y) \geq \delta.
\label{sep} \end{equation}
 Let  $\i=(i_1,i_2, \ldots), \l=(l_1,l_2, \ldots) \in \Sigma$ with $\i|n \neq \l|n$. In particular there exists $0 \leq m \leq n-1$ such that $\i|m=\l|m$ and $i_{m+1} \neq l_{m+1}$. We write $\i|n=\i|m\j$ and $\l|n=\i|m \k$. Then for all $\a, \b \in [0,1]$,
\begin{eqnarray*}
d(S_{\i|n}(\a), S_{\l|n}(\b))&=&d(S_{\i|m}(S_{\j}(\a)), S_{\i|m}(S_\k(\b))) \\
&\gtrsim& d(S_\j(\a),S_\k(\b))\sup_{\c_3 \in [0,1]^2}\alpha_2(D_{\c_3} S_{\i|m}) \\
& \gtrsim& \alpha_2(D_{\Pi(\sigma^{n-1}\i)} S_{\i|n-1})
\end{eqnarray*}
where the second inequality follows by \eqref{cyl contr} and the final inequality follows by \eqref{sep} (since $\j$ and $\k$ begin with different digits) and Lemma \ref{svboundeddistortion}.
\end{proof}

We are now in a position to be able to prove Theorem \ref{ly theorem}, our main result. We do so by establishing both the corresponding lower and upper bounds for the local dimension of $\mu$ at $\Pi(\i)$ for $\i\in\Sigma$ belonging to a set of full $m$-measure. It is worth noting that only the lower bound requires Lemma \ref{sep-lemma} and as such this is the only bound that requires the SSC.

\begin{proof}[Proof of Theorem \ref{ly theorem}]
 Let $\i\in\Sigma$ belong to the set of full $m$-measure for which \eqref{entropy1}, \eqref{lyapunov1} and \eqref{lyapunov2} hold.  Let
\[
\eta:=\sup_{i\in\mathcal{I}, \a, \b\in[0, 1]^2}\left\{\frac{|g_{i,y}(\a)|}{|f_{i,x}(\b)|}\right\}
\]
and note that by the domination condition from Definition \ref{ifsdef} $\eta<1$. Using Lemma \ref{Singular Value Bound}, applying the chain rule to $g_{\i|n,y}(\Pi(\sigma^n \i))$ and $f_{\i|n,x}(\b)$ for each $n\in\mathbb{N}$ and pairing off appropriate terms we get  
\[
\frac{\alpha_2\left(D_{\Pi(\sigma^{n} \i)}S_{\i|n}\right)}{\min_{\b \in [0,1]^2} \alpha_1\left(D_{{\bf b}}S_{\i|n}\right)}\leq\frac{M^2|g_{\i|n,y}(\Pi(\sigma^n \i))|}{\min_{\b \in [0,1]^2} |f_{\i|n,x}(\b)|}\leq M^2\eta^n
\rightarrow 0.
\]
Similarly,
\[
\frac{\alpha_2\left(D_{\Pi(\sigma^{n} \i)}S_{\i|n}\right)}{\max_{\b \in [0,1]^2} \alpha_1\left(D_{{\bf b}}S_{\i|n}\right)}\leq M^2\eta^n\rightarrow 0.
\]
Define the sequences
\begin{equation} \label{sequences}
r_n(\i)=\frac{M \theta\alpha_2\left(D_{\Pi(\sigma^{n} \i)}S_{\i|n}\right)}{\min_{\b \in [0,1]^2} \alpha_1\left(D_{{\bf b}}S_{\i|n}\right)}\; \textnormal{ and } \;  r_n'(\i)=\frac{M^{-1}\alpha_2\left(D_{\Pi(\sigma^{n} \i)}S_{\i|n}\right)}{\max_{\b \in [0,1]^2} \alpha_1\left(D_{{\bf b}}S_{\i|n}\right)},
\end{equation}
and observe that both $r_n(\i)$ and $r_n'(\i)$ converge to 0 uniformly over all $\i \in \Sigma$.
Hence we can also assume that $\i \in \Sigma$ belongs to the set of full measure which satisfies \eqref{subsequence1} and \eqref{subsequence2} for the sequences $r_n(\i)$ and $r_n'(\i)$.

For a given $\a=(a_1,a_2)\in [0,1]^2$ and $r>0$  write 
\[
Q_2(\a, r)=\left(a_1-\frac{r}{2},a_1+\frac{r}{2}\right)\times \left(a_2-\frac{r}{2}, a_2+\frac{r}{2}\right).
\] 
Write $\textbf{x}=\Pi(\i)$, let $n\in\mathbb{N}$ and consider the square $Q_2\left(\textbf{x}, \theta \alpha_2\left(D_{\Pi(\sigma^{n} \i)}S_{\i|n}\right)\right)$. By Lemma \ref{sep-lemma} note that $Q_2\left(\textbf{x}, \theta \alpha_2\left(D_{\Pi(\sigma^{n} \i)}S_{\i|n}\right)\right)$ intersects only the cylinder $S_{\i|n}([0, 1]^2)$, therefore it is easy to see that
\[
Q_2\left(\textbf{x}, \theta\alpha_2\left(D_{\Pi(\sigma^{n} \i)}S_{\i|n}\right)\right)\cap F\subseteq B_n(\textbf{x}, \theta\alpha_2\left(D_{\Pi(\sigma^{n} \i)}S_{\i|n}\right)).
\]
Hence by Lemma \ref{mu bound}, 
\begin{align}
&\mu\left(Q_2\left(\textbf{x}, \theta \alpha_2\left(D_{\Pi(\sigma^{n} \i)}S_{\i|n}\right)\right)\right)\nonumber\\
&\leq Lm([\i|n]) \pi(\mu)\left(Q_1\left(\pi(\Pi(\sigma^n \i)),\frac{M\theta\alpha_2\left(D_{\Pi(\sigma^{n} \i)}S_{\i|n}\right)}{\min_{\b \in [0,1]^2} \alpha_1\left(D_{{\bf b}}S_{\i|n}\right)}\right)\right). \label{mu q bound}
\end{align}
Consider the subsequence $(n_k)_{k \in \N}$ guaranteed by applying Lemma \ref{subsequence} to the sequence $r_n(\i)$.
Applying the chain rule we have
\begin{align*}
\alpha_2\left(D_{\Pi(\sigma^{n+1} \i)}S_{\i|n+1}\right)&=\alpha_2\left(D_{S_{i_{n+1}}\Pi(\sigma^{n+1} \i)}S_{\i|n}D_{\Pi(\sigma^{n+1} \i)}S_{i_{n+1}}\right)\\
&\leq \alpha_2\left(D_{S_{i_{n+1}}\Pi(\sigma^{n+1} \i)}S_{\i|n}\right)\alpha_1\left(D_{\Pi(\sigma^{n+1} \i)}S_{i_{n+1}}\right)\\
&<\alpha_2\left(D_{S_{i_{n+1}}\Pi(\sigma^{n+1} \i)}S_{\i|n}\right)\\
&= \alpha_2\left(D_{\Pi(\sigma^{n} \i)}S_{\i|n}\right),
\end{align*}
where in the first inequality we have used that $\alpha_2(AB) \leq \alpha_2(A)\alpha_1(B)$ for $2 \times 2$ matrices $A$ and $B$. This implies that the null subsequence $\left(\theta \alpha_2\left(D_{\Pi(\sigma^{n_k} \i)}S_{\i|n_k}\right)\right)_{k\in\mathbb{N}}$ is strictly decreasing. Hence for any $r>0$ sufficiently small we can choose $k\in\mathbb{N}$ sufficiently large so that
\begin{equation}\label{k-thing}
\theta  \alpha_2\left(D_{\Pi(\sigma^{n_{k+1}} \i)}S_{\i|n_{k+1}}\right)\leq r< \theta \alpha_2\left(D_{\Pi(\sigma^{n_k} \i)}S_{\i|n_k}\right).
\end{equation}

Let $t=\dim \pi(\mu)$ and $\varepsilon>0$. Let $r>0$ be sufficiently small so that $k \in \N$ that satisfies \eqref{k-thing} is sufficiently large that \eqref{subsequence1} holds for $\varepsilon$. Then, using \eqref{lyap-alt}, \eqref{mu q bound} and \eqref{subsequence1} we get 
\begin{align*}
&\frac{\log \mu\left(Q_2(\textbf{x}, r)\right)}{\log r}\\
&\geq \frac{\log \mu\left(Q_2\left(\textbf{x}, \theta \alpha_2\left(D_{\Pi(\sigma^{n_k} \i)}S_{\i|n_k}\right)\right)\right)}{\log \left(\theta \alpha_2\left(D_{\Pi(\sigma^{n_{k+1}} \i)}S_{\i|_{n_{k+1}}}\right)\right)}\\
&\geq \frac{\log Lm([\i|n_k]) + \log\pi(\mu)\left(Q_1\left(\pi(\Pi(\sigma^{n_k} \i)), \frac{M\theta\alpha_2\left(D_{\Pi(\sigma^{n_k} \i)}S_{\i|n_k}\right)}{\min_{\b \in [0,1]^2} \alpha_1\left(D_{{\bf b}}S_{\i|n_k}\right)}\right)\right)}{\log \left(\theta \alpha_2\left(D_{\Pi(\sigma^{n_{k+1}} \i)}S_{\i|_{n_{k+1}}}\right)\right)}\\
&\geq \frac{\log Lm([\i|n_k]) + (t-\varepsilon)\log\left(\frac{M\theta\alpha_2\left(D_{\Pi(\sigma^{n_k} \i)}S_{\i|n_k}\right)}{\min_{\b \in [0,1]^2} \alpha_1\left(D_{{\bf b}}S_{\i|n_k}\right)}\right)}{\log \left(\theta \alpha_2\left(D_{\Pi(\sigma^{n_{k+1}} \i)}S_{\i|_{n_{k+1}}}\right)\right)}\\
&=\frac{\frac{1}{n_k}\log L+\frac{1}{n_k}\log m([\i|n_k]) +\frac{(t-\varepsilon)}{n_k}\log\left(\frac{\alpha_2\left(D_{\Pi(\sigma^{n_k} \i)}S_{\i|n_k}\right)}{\min_{\b \in [0,1]^2} \alpha_1\left(D_{{\bf b}}S_{\i|n_k}\right)}\right)+\frac{(t-\varepsilon)}{n_k}\log M\theta}{\frac{1}{n_k}\log \theta  +\frac{1}{n_{k+1}}\frac{n_{k+1}}{n_k}\log \left(\alpha_2\left(D_{\Pi(\sigma^{n_{k+1}} \i)}S_{\i|_{n_{k+1}}}\right)\right)}\\
&\rightarrow\frac{h(\mu)+(t-\varepsilon)(\chi_2(\mu)-\chi_1(\mu))}{\chi_2(\mu)}
\end{align*}
as $r\rightarrow 0$ (so $k\rightarrow\infty$). Since $\varepsilon> 0$ was arbitrary, the lower bound is complete.

We now establish the corresponding upper bound. We begin by estimating 
\[
\sup_{(a_1,a_2),(b_1,b_2) \in B_n(\mathbf{x}, \alpha_2(D_{\Pi(\sigma^n(\i))} S_{\i|n})) } |a_2-b_2|.
\]
for each $n\in\mathbb{N}$. For some $a,b \in [0,1]$ with the property that
\begin{equation}\label{f alpha2 bound}
|f_{\i|n}(b)-f_{\i|n}(a)|\leq\alpha_2\left(D_{\Pi(\sigma^{n} \i)}S_{\i|n}\right)
\end{equation}
we can write
\[
\sup_{(a_1,a_2),(b_1,b_2) \in B_n(\mathbf{x}, \alpha_2(D_{\Pi(\sigma^n(\i))} S_{\i|n}))} |a_2-b_2|=|g_{\i|n}(b, 1)-g_{\i|n}(a, 0)|.
\]
Note that
\begin{align*}
|g_{\i|n}(b, 1)-g_{\i|n}(a, 0)|&\leq |g_{\i|n}(b, 1)-g_{\i|n}(a, 1)| + |g_{\i|n}(a, 1)-g_{\i|n}(a, 0)|\\
& \leq C|f_{\i|n}(b)-f_{\i|n}(a)| + |g_{\i|n,y}(\c)|
\end{align*}
for some $\c \in [0,1]^2$ where we have used \eqref{boundedheight} and the mean value theorem. Thus it follows from \eqref{f alpha2 bound} and Lemma \ref{boundeddistortion} that
\begin{align*}
|g_{\i|n}(b, 1)-g_{\i|n}(a, 0)|&\leq C\alpha_2\left(D_{\Pi(\sigma^{n} \i)}S_{\i|n}\right) + A\alpha_2\left(D_{\Pi(\sigma^{n} \i)}S_{\i|n}\right)\\
&=(A+C)\alpha_2\left(D_{\Pi(\sigma^{n} \i)}S_{\i|n}\right).
\end{align*}

It is now easy to see that
\[
B_n(\textbf{x}, \alpha_2\left(D_{\Pi(\sigma^{n} \i)}S_{\i|n}\right))\cap F\subseteq Q_2\left(\textbf{x}, 2(A+C)\alpha_2\left(D_{\Pi(\sigma^{n} \i)}S_{\i|n}\right)\right)\cap F.
\]
Thus Lemma \ref{mu bound} implies
\begin{align}
&\mu\left(Q_2\left(\textbf{x}, 2(A+C)\alpha_2\left(D_{\Pi(\sigma^{n} \i)}S_{\i|n}\right)\right)\right)\nonumber\\
&\geq L^{-1}m([\i|n]) \pi(\mu)\left(Q_1\left(\pi(\Pi(\sigma^{n} \i)), \frac{M^{-1}\alpha_2\left(D_{\Pi(\sigma^{n} \i)}S_{\i|n}\right)}{\max_{\b \in [0,1]^2} \alpha_1\left(D_{{\bf b}}S_{\i|n}\right)}\right)\right)\label{mu q bound 2}. 
\end{align}
Consider the subsequence $(n_k)_{k \in \N}$ guaranteed by applying Lemma \ref{subsequence} for $r_n'(\i)$, which was defined in \eqref{sequences}. Since $\left(\alpha_2\left(D_{\Pi(\sigma^{n} \i)}S_{\i|n}\right)\right)_{n\in\mathbb{N}}$ is strictly decreasing and null, for any $r>0$ sufficiently small we can choose $k\in\mathbb{N}$ sufficiently large so that
\begin{equation} \label{k-thing2}
2(A+C)\alpha_2\left(D_{\Pi(\sigma^{n_{k+1}} \i)}S_{\i|_{n_{k+1}}}\right)\leq r< 2(A+C)\alpha_2\left(D_{\Pi(\sigma^{n_k} \i)}S_{\i|n_k}\right).
\end{equation}

Let $\varepsilon>0$. Let $r>0$ be sufficiently small so that $k \in \N$ that satisfies \eqref{k-thing2} is sufficiently large that \eqref{subsequence2} holds for $\varepsilon$. Therefore by using \eqref{lyap-alt}, \eqref{mu q bound 2} and \eqref{subsequence2} we get
\begin{align*}
&\frac{\log \mu\left(Q_2(\textbf{x}, r)\right)}{\log r}\\
&\leq \frac{\log \mu\left(Q_2\left(\textbf{x}, 2(A+C)\alpha_2\left(D_{\Pi(\sigma^{n_{k+1}} \i)}S_{\i|_{n_{k+1}}}\right)\right)\right)}{\log \left(2(A+C)\alpha_2\left(D_{\Pi(\sigma^{n_k} \i)}S_{\i|n_k}\right)\right)}\\
&\leq \frac{\log L^{-1}m([\i|_{n_{k+1}}]) + \log\pi(\mu)\left(Q_1\left(\pi(\Pi(\sigma^{n_{k+1}} \i)),  \frac{M^{-1}\alpha_2\left(D_{\Pi(\sigma^{n_{k+1}} \i)}S_{\i|_{n_{k+1}}}\right)}{\max_{\b \in [0,1]^2} \alpha_1\left(D_{{\bf b}}S_{\i|n_{k+1}}\right)}\right)\right)}{\log\left(2(A+C)\alpha_2\left(D_{\Pi(\sigma^{n_k} \i)}S_{\i|n_k}\right)\right)}\\
&\leq \frac{\log L^{-1}m([\i|_{n_{k+1}}]) + (t+\varepsilon)\log\left(\frac{M^{-1}\alpha_2\left(D_{\Pi(\sigma^{n_{k+1}} \i)}S_{\i|_{n_{k+1}}}\right)}{\max_{\b \in [0,1]^2} \alpha_1\left(D_{{\bf b}}S_{\i|n_{k+1}}\right)}\right)}{\log\left(2(A+C)\alpha_2\left(D_{\Pi(\sigma^{n_k} \i)}S_{\i|n_k}\right)\right)}\\
&=\frac{\frac{1}{n_{k+1}}\log L^{-1}+\frac{1}{n_{k+1}}\log m([\i|_{n_{k+1}}]) +\frac{(t+\varepsilon)}{n_{k+1}}\log\left(\frac{\alpha_2\left(D_{\Pi(\sigma^{n_{k+1}} \i)}S_{\i|_{n_{k+1}}}\right)}{\max_{\b \in [0,1]^2} \alpha_1\left(D_{{\bf b}}S_{\i|n_{k+1}}\right)}\right)+\frac{(t+\varepsilon)}{n_{k+1}}\log M^{-1}}{\frac{1}{n_{k+1}}\log 2(A+C) +\frac{1}{n_k}\frac{n_k}{n_{k+1}}\log \left(\alpha_2\left(D_{\Pi(\sigma^{n_k} \i)}S_{\i|n_k}\right)\right)}\\
&\rightarrow\frac{h(\mu)+(t+\varepsilon)(\chi_2(\mu)-\chi_1(\mu))}{\chi_2(\mu)}
\end{align*}
as $r\rightarrow 0$ (so $k\rightarrow\infty$). Since $\varepsilon> 0$ was arbitrary, the upper bound follows.
\end{proof}

\begin{samepage}

\subsection*{Acknowledgements}

LDL was supported by an \emph{EPSRC Doctoral Training Grant} (EP/N509759/1). NJ was supported by an \emph{EPSRC Standard Grant} (EP/R015104/1). The authors thank Kenneth Falconer and Jonathan Fraser for making several helpful comments on the manuscript.

\end{samepage}

Natalia Jurga, School of Mathematics \& Statistics, University of St Andrews, St Andrews, KY16 9SS, UK
\textit{E-mail address}:\ \url{naj1@st-andrews.ac.uk}

Lawrence D. Lee, School of Mathematics \& Statistics, University of St Andrews, St Andrews, KY16 9SS, UK
\textit{E-mail address}:\ \url{ldl@st-andrews.ac.uk}


\begin{bibdiv}
\begin{biblist}
	
\bib{bedford}{article}{
   author={Bedford, Tim},
   title={Applications of dynamical systems theory to fractals---a study of
   cookie-cutter Cantor sets},
   conference={
      title={Fractal geometry and analysis},
      address={Montreal, PQ},
      date={1989},
   },
   book={
      series={NATO Adv. Sci. Inst. Ser. C Math. Phys. Sci.},
      volume={346},
      publisher={Kluwer Acad. Publ., Dordrecht},
   },
   date={1991},
   pages={1--44},
}

\bib{bk}{article}{
   author={B\'{a}r\'{a}ny, Balazs},
   author={K\"{a}enm\'{a}ki, Antti},
   title={Ledrappier-Young formula and exact dimensionality
of self-affine measures},
journal={Advances in Mathematics},
  volume={318},
  pages={88--129},
  year={2017},
}

\bib{bkm}{article}{
   author={B\'{a}r\'{a}ny, Balazs},
   author={K\"{a}enm\'{a}ki, Antti},
   author={Morris, Ian},
   title={Domination, almost additivity, and thermodynamic formalism for planar matrix cocycles},
journal={Israel Journal of Mathematics (to appear)},
}

\bib{f techniques}{book}{
   author={Falconer, Kenneth}
   title={Techniques in Fractal Geometry},
   publisher={Wiley},
   year={1997}
}	
	
\bib{ffl}{article}{
   author={Falconer, Kenneth},
   author={Fraser, Jonathan},
   author={Lee, Lawrence},
   title={$L^q$ spectra of measures on planar non-conformal attractors},
journal={Ergodic Theory and Dynamical Systems (to appear)},
}

\bib{feng}{article}{
   author={Feng, D.-J.},
   title={Dimension of invariant measures for affine iterated function systems},
eprint={https://arxiv.org/abs/1901.01691}
}

\bib{fh}{article}{
   author={Feng, D.-J.},
   author={Hu, Huyi},
   title={Dimension theory of iterated function systems},
journal={Communications on Pure and Applied Mathematics: A Journal Issued by the Courant Institute of Mathematical Sciences},
  volume={62},
  number={11},
  pages={1435--1500},
  year={2009},
}

\bib{fs}{article}{
   author={Feng, D.-J.},
   author={Simon, Karoly},
   title={Dimension estimates for $C^1$ iterated function
systems and repellers. Part I},
journal={Preprint, available at https://arxiv.org/abs/2007.15320},
  year={2020},
}

\bib{fjj}{article}{
   author={Fraser, Jonathan},
   author={Jordan, Thomas},
   author={Jurga, Natalia},
   title={Dimensions of equilibrium measures
on a class of planar self-affine sets},
journal={Journal of Fractal Geometry},
  volume={7},
  number={1},
  pages={87--111},
  year={2020},
}

\bib{hu}{article}{
   author={Hu, H},
   title={Dimensions of invariant sets of expanding maps},
journal={Communications in Mathematical Physics},
  volume={176},
  pages={307-320},
  year={1996},
}

\bib{hut}{article}{
   author={Hutchinson, J.E.},
   title={Fractals and self-similarity},
journal={Indiana University Mathematics  Journal},
  volume={30},
  pages={713-747},
  year={1981},
}

\bib{ly1}{article}{
   author={Ledrappier, Francois},
   author={Young, Lai-Sang},
   title={The metric entropy of diffeomorphisms. I. Characterization of measures satisfying Pesin's entropy formula},
journal={Annals of Mathematics},
  volume={122},
  number={3},
  pages={509--539},
  year={1985},
}

\bib{ly2}{article}{
   author={Ledrappier, Francois},
   author={Young, Lai-Sang},
   title={The metric entropy of diffeomorphisms. II. Relations between entropy, exponents and dimension},
journal={Annals of Mathematics},
  volume={122},
  number={3},
  pages={540--574},
  year={1985},
}

\bib{pesin}{article}{
   author={Cao, Yongluo},
   author={Pesin, Yakov},
   author={Zhao, Yun},
   title={Dimension estimates for non-conformal repellers and continuity of
   sub-additive topological pressure},
   journal={Geometric and Functional Analysis},
   volume={29},
   date={2019},
   number={5},
   pages={1325--1368}
}

\end{biblist}
\end{bibdiv}
\end{document}